\documentclass[12pt,letterpaper,reqno]{amsart}

\usepackage{times}
\usepackage[T1]{fontenc}
\usepackage{mathrsfs}
\usepackage{latexsym}
\usepackage{amsmath,amsfonts,amsthm,amssymb,amscd}
\input amssym.def
\input amssym.tex

\newtheorem{thm}{Theorem}[section]
\newtheorem{lem}[thm]{Lemma}

\newtheorem{cor}[thm]{Corollary}

\newtheorem{rek}[thm]{Remark}

\newcommand\bi{\begin{itemize}}
\newcommand\ei{\end{itemize}}

\newcommand\ben{\begin{enumerate}}
\newcommand\een{\end{enumerate}}
\newcommand{\ncr}[2]{{#1 \choose #2}}

\newcommand{\N}{\mathbb{N}}
\newcommand{\R}{\mathbb{R}}

\newcommand{\Z}{\mathbb{Z}}
\newcommand{\Q}{\mathbb{Q}}

\newcommand{\twocase}[5]{#1 \begin{cases} #2 & \text{{\rm #3}}\\ #4
&\text{{\rm #5}} \end{cases}  }

\newcommand\be{\begin{equation}}
\newcommand\ee{\end{equation}}
\newcommand\bea{\begin{eqnarray}}
\newcommand\eea{\end{eqnarray}}

\newcommand{\E}{{\mathbb E}} 
\newcommand{\mel}{\mathcal{M}}

\renewcommand{\Re}{{\mathfrak{Re}}}

\newcommand{\unif}[1]{{\rm Unif}(0,#1)}
\newcommand{\expd}[1]{{\rm Exp}(#1)}
\newcommand{\err}[1]{{\rm Err}\left(#1\right)}

\newcommand{\meijer}[9]{G^{#1, #2}_{#3, #4}\left(#5 \Big| {#6, \dots, #7 \atop #8, \dots, #9}\right)}

\newcommand{\meijerone}[7]{G^{#1, #2}_{#3, #4}\left(#5 \Big| {#6 \atop #7}\right)}

\newcommand{\Bmeijer}[9]{B^{#1, #2}_{#3, #4}\left(#5 \Big| {#6, \dots, #7 \atop #8, \dots, #9}\right)}

\numberwithin{equation}{section}

\begin{document}

\title{Chains of distributions, hierarchical Bayesian models and Benford's Law}

\author{Dennis Jang}\email{Dennis\underline{\ }Jang@brown.edu}
\author{Jung Uk Kang}\email{Jung\underline{\ }Uk\underline{\ }Kang@brown.edu}
\author{Alex Kruckman}\email{Alex\underline{\ }Kruckman@brown.edu}
\author{Jun Kudo}\email{Jun\underline{\ }Kudo@brown.edu}
\address{Department of Mathematics, Brown University, Providence, RI $02912$}
\author{Steven J. Miller}\email{Steven.J.Miller@williams.edu}
\address{Department of Mathematics, Brown University, Providence, RI $02912$ and Department of Mathematics and Statistics, Williams College, Williamstown, MA 01267}


\date{\today}

\begin{abstract} Kossovsky recently conjectured that the distribution of leading digits of a chain of probability distributions converges to Benford's law as the length of the chain grows. We prove his conjecture in many cases, and provide an interpretation in terms of products of independent random variables and a central limit theorem. An interesting consequence is that in hierarchical Bayesian models priors tend to satisfy Benford's Law as the number of levels of the hierarchy increases, which allows us to develop some simple tests (based on Benford's law) to test proposed models. We give explicit formulas for the error terms as sums of Mellin transforms, which converges extremely rapidly as the number of terms in the chain grows. We may interpret our results as showing that certain Markov chain Monte Carlo processes are rapidly mixing to Benford's law.
\end{abstract}

\maketitle


\section{Introduction}

The distribution of leading digits of numbers in data sets has intrigued researchers for over 100 years. Using scientific notation (base $B$), for any $x > 0$ we may write $x = M_B(x) B^k$, where $k\in\Z$ and $M_B(x)$ is the mantissa of $x$ base $B$. We say the data follows Benford's law if the probability of having a mantissa of at most $s$ is $\log_B s$. This implies that the probability of observing a first digit of $d$ base $B$ is $\log_B\left(1+1/d\right)$; in particular, about 30\% of the time the first digit is 1 base 10 (and not 11\% as one might naively guess). Many systems are known to satisfy Benford's law. Examples include recurrence relations \cite{BrDu}, $n!$ and $\ncr{n}{k}$ ($0 \le k \le n$) \cite{Dia}, iterates of
power, exponential and rational maps \cite{BBH,Hi2}, values of
$L$-functions near the critical line and characteristic polynomials
of random matrix ensembles \cite{KonMi}, iterates of the $3x+1$ Map
\cite{KonMi,LS} and differences of order statistics \cite{MN}, to name a few. In addition to arising in a variety of mathematical settings, Benford's Law surfaces in diverse fields, from atomic physics \cite{P} to biology \cite{CLTF} to geology \cite{NM} to the stock market \cite{Ley}. Applications range from detecting fraud in accounting \cite{Nig1,Nig2} and social sciences \cite{Me} to determining optimal ways to store numbers (see page 255 of \cite{Knu} and \cite{BH}). See \cite{Hi1,Rai} for a description and history of the subject, and \cite{Hu} for a detailed bibliography of the field. In this paper we show how Benford's law arises in chains of probability distributions and hierarchical Bayesian models. This allows us to construct tests (based on Benford's law) of certain models. We may interpret our results as saying that in many Markov chain Monte Carlo problems, the stationary distribution of first digits is Benford's law, and the chain has rapid mixing (i.e., few iterations are required to have excellent agreement with Benford's law).

Since the early work of Newcomb \cite{New} and Benford \cite{Ben}, there have been numerous theoretical advances as to why various data sets and operations yield Benford behavior. One reason for the immense amount of interest generated by this law is the observation that, in many cases, combining two data sets yields a new set which is closer to Benford's law (see for example \cite{Ha}). A common example is street addresses. If one studies the distribution of leading digits on a long street, the result is clearly non-Benford; depending on the length of the street, the probability of a first digit of 1 can oscillate between $1/9$ and $5/9$. However, if we consider many streets and amalgamate the data (as Benford \cite{Ben} did), the result is quite close to Benford's law. We may interpret the above as first choosing a street length from some distribution, so the street addresses say are integers in $[1,X]$ for some random variable $X$. Then for each choice of $X$ we study the distribution of the leading digits on that street, and then calculate the expected frequencies as $X$ varies.


In \cite{Ko}, Kossovsky suggested such an interpretation and proposed that generalizations of the above procedure will rapidly lead to convergence to Benford behavior. Explicitly, he studied the distribution of leading digits of chained probability distributions, and conjectured that as the length of the chain increases then the behavior tends to Benford's law. In this note we quantify and prove some of his conjectures; see \cite{Ko} for a complete description of his investigations. Let $\mathcal{D}_i(\theta)$ denote a one-parameter distribution with parameter $\theta$ and density function $f_{\mathcal{D}_i(\theta)}$; thus by $X\sim \mathcal{D}_i(\theta)$ we mean \be {\rm Prob}(X \in [a,b]) \ = \ \int_a^b f_{\mathcal{D}_i(\theta)}(x)dx. \ee We create a chain of random variables as follows. Let $p: \N \to \N$. Let $X_1 = \mathcal{D}_{p(1)}(1)$ and define $X_m$ inductively by $X_m \sim \mathcal{D}_{p(m)}(X_{m-1})$. Computer simulations and other considerations led Kossovsky to conjecture that if our underlying distributions are `nice', then as $n\to\infty$ the distribution of the leading digits of $X_n$ converges to Benford's law, and further that if $X_1$ is Benford then $X_n$ is Benford. Note that our example of street addresses is just a special case with a chain length of two and uniform distributions. Another way of stating our results is that for certain Markov chain Monte Carlo processes, Benford's law is absorbing for the distribution of first digits (and in fact the system is rapidly mixing as well).

We prove his claims in several cases, providing a partial answer to which distributions are `nice'.\footnote{The conjecture may fail if we chain arbitrary parameters of arbitrary distributions. A good test case is to consider chaining the shape parameter $\gamma$ of a Weibull distribution: $f(x) = \gamma x^{\gamma - 1} \exp(-x^\gamma)$ for $x\ge 0$. The difficulty with numerics here is that very quickly we end up with a shape parameter very small (say less than $10^{-20}$), and thus the numerics become suspect.} Before stating our results, we first discuss some important consequences. Returning to our street example, we see we may reformulate it in terms of a Bayesian model (see \cite{Ber} for more details). In Bayesian models we have some data (say $x$) whose values depend on a parameter (say $\beta$, called the prior). Thus there are two densities, that of the data (which depends on $\beta$) and that of the prior. In our situation, $x$ would be the street address, drawn from a uniform distribution on say $[1,\beta]$, and then $\beta$ would be drawn from some distribution modeling how street lengths are distributed. One can of course consider more involved models where the prior depends on a hyperparameter drawn from a different distribution (and so on). These are called hierarchical Bayesian models, and in this setting we again encounter chains of distribution, where the number of chains is basically the number of levels.

One of the major problems in Bayesian theory is to justify the choice of the prior. Many ideas have been proposed (for example, Jeffrey's prior, conjugate priors, empirical Bayes, hierarchical models). In putting priors on hyperparameters, we often make our prior more ``diffuse'', so to speak, or less informative. Our main result says that, in many cases, a non-informative prior in this hierarchical sense leads to sample data closely approximating Benford's Law; further, in many situations a Benford prior might be the true non-informative prior, rather than classic approaches which are essentially variants on the uniform distribution. Our results can thus be used as a data integrity check in this situation.

We introduce some notation and then state our main results. By $\err{z}$ we mean an error at most $z$ in absolute value. Let $f(x)$ be a continuous real-valued function on $[0,\infty)$. We define its Mellin transform, $(\mathcal{M}f)(s)$, by \be (\mathcal{M}f)(s) \ = \ \int_0^\infty f(x) x^s \frac{dx}{x}. \ee Note $(\mathcal{M}f)(s) = \E[x^{s-1}]$, and thus results about expected values translate to results on Mellin transforms; for example, $(\mathcal{M}f)(1) = 1$ for any distribution supported on $[0,\infty)$.

If $g(s)$ is an analytic function for $\Re(s) \in (a,b)$ such that $g(c+iy)$ tends to zero uniformly as $|y| \to \infty$ for any $c \in (a,b)$, then the inverse Mellin transform, $(\mathcal{M}^{-1}g)(x)$, is given by \be (\mathcal{M}^{-1}g)(x) \ = \ \frac1{2\pi i} \int_{c-i\infty}^{c+i\infty} g(s) x^{-s} ds \ee (provided that the integral converges absolutely). If we set $g(s) = (\mathcal{M}f)(s)$ then $f(x) = (\mathcal{M}^{-1}g)(x)$. We define the convolution of two functions $f_1$ and $f_2$ by \be (f_1 \star f_2)(x) \ = \ \int_0^\infty f_2\left(\frac{x}{t}\right) f_1(t) \frac{dt}{t} \ = \ \int_0^\infty f_1\left(\frac{x}{t}\right) f_2(t) \frac{dt}{t}. \ee The Mellin convolution theorem states that \be (\mathcal{M}(f_1\star f_2))(s) \ = \ (\mel f_1)(s) \cdot (\mel f_2)(s), \ee which by induction\footnote{As $(\mathcal{M}f)(s) = \E[x^{s-1}]$, we may re-interpret the following in terms of products of independent random variables; see also Remark \ref{rek:interpretprodrandvar}.} gives \be (\mathcal{M}(f_1\star \cdots \star f_n))(s) \ = \ (\mel f_n)(s) \cdots (\mel f_n)(s). \ee See Appendix 2 of \cite{Pa} for an enumeration of properties of the Mellin transform.\footnote{If we let $x = e^{2\pi u}$ and $s=\sigma - i\xi$, then $(\mel f)(\sigma - i\xi) = 2\pi \int_{-\infty}^\infty \left( f(e^{2\pi u}) e^{2\pi \sigma u}\right) e^{-2\pi i u \xi} du$, which is the Fourier transform of $g(u) = 2\pi f(e^{2\pi u})e^{2\pi\sigma u}$. The Mellin and Fourier transforms as thus related; in fact, it is this logarithmic change of variables which explains why both enter into Benford's law problems. For proofs of the Mellin transform properties one can therefore just mimic the proofs of the corresponding statements for the Fourier transform; a good reference is \cite{SS}.}

Our main results are the following:

\begin{thm}\label{thm:main} Let $\{\mathcal{D}_i(\theta)\}_{i\in I}$ be a collection of one-parameter distributions with associated densities $f_{\mathcal{D}_i(\theta)}$ which vanish outside of $[0,\infty)$. Let $p:\N \to I$, $X_1 \sim \mathcal{D}_{p(1)}(1)$, $X_m \sim \mathcal{D}_{p(m)}(X_{m-1})$, and assume \ben \item for each $m \ge 2$, \bea\label{eq:keydensitylink} f_m(x_m) \ = \ \int_0^\infty f_{\mathcal{D}_{p(m)}(1)}\left(\frac{x_m}{x_{m-1}}\right) f_{m-1}(x_{m-1}) \frac{dx_{m-1}}{x_{m-1}} \eea where $f_m$ is the density of the random variable $X_m$ (see Lemma \ref{lem:firstcondthm} for examples where this condition is satisfied); \item we have \be\label{eq:summeltransformsnozero} \lim_{n\to\infty} \sum_{\ell = -\infty \atop \ell \neq 0}^\infty \prod_{m=1}^n (\mathcal{M} f_{\mathcal{D}_{p(m)}(1)})\left(1-\frac{2\pi i \ell}{\log B}\right) \ = \ 0. \ee \een Then as $n\to\infty$ the distribution of leading digits of $X_n$ tends to Benford's law. Further, the error is a nice function of the Mellin transforms. Explicitly, if $Y_n = \log_B X_n$, then \bea & & \left|{\rm Prob}(Y_n \bmod 1 \in [a,b]) - (b-a)\right| \nonumber\\ & & \ \ \ \ \ \ \ \le \ (b-a) \cdot \left|\sum_{\ell = -\infty \atop \ell \neq 0}^\infty \prod_{m=1}^n (\mathcal{M} f_{\mathcal{D}_{p(m)}(1)})\left(1-\frac{2\pi i \ell}{\log B}\right)\right|. \eea If $I$ is finite and all densities are continuous, then the second condition holds.
\end{thm}

The second condition in Theorem \ref{thm:main} is extremely weak, and is typically satisfied in all examples of interest. For example, assume $I$ is finite and all the densities are continuous. Then for $\ell \neq 0$ we have rapid decay (in $\ell$) of $\left|(\mathcal{M} f_{\mathcal{D}_{p(m)}(1)})\left(1-\frac{2\pi i \ell}{\log B}\right)\right|$; this is because our expression is equivalent to taking the Fourier transform of a related, continuous function at $\ell/\log B$, which by the Riemann-Lebesgue lemma tends to zero as $|\ell| \to \infty$. With some work, we can construct a pathological infinite family of distinct densities where this product condition fails; see \cite{MN} for the details. Note that for any density $f$ we have $(\mel f)(1) = 1$. This is why in \eqref{eq:summeltransformsnozero} we sum only over $\ell \neq 0$; the $\ell = 0$ term is always 1, and gives the main term term. Frequently this sum tends to zero very rapidly with $n$; we give some explicit examples in \S\ref{sec:examples}.

The first condition is more serious, and thus we give a few non-trivial examples where it holds.

\begin{lem}\label{lem:firstcondthm} Assume the density $f_{\mathcal{D}_{p(m)}(\theta)}(x) = \theta^{-1} f(x/\theta)$ for some $f$ (with antiderivative $F$). Let $X_{m-1}$ have density $f_{m-1}$ and let $X_m \sim \mathcal{D}_{p(m)}(X_{m-1})$. Then \eqref{eq:keydensitylink} is satisfied for $X_m$. Examples include \bi \item Let $\mathcal{D}_{\rm unif}(\theta)$ be the uniform distribution on $[0,\theta]$ (thus $f_{\mathcal{D}_{\rm unif}(\theta)}(x) = 1/\theta$ for $x \in [0,\theta]$ and $0$ otherwise); \item Let $\mathcal{D}_{\exp}(\theta)$ be the exponential distribution with parameter $\theta$ (thus $f_{\mathcal{D}_{\exp}(\theta)}(x) = \theta^{-1} \exp(-x/\theta)$ for $x \ge 0$ and $0$ otherwise); \item $\mathcal{D}_{|{\rm gauss}|}(\theta)$ be the density of $|W|$ where $W \sim N(0,\theta/\sqrt{2})$ (thus $f_{\mathcal{D}_{|{\rm gauss}|}(\theta)}(x)$ $=$\\ $(2/\sqrt{\pi\theta^2}) \exp(-(x/\theta)^2)$ if $x \ge 0$ and $0$ otherwise). \ei Thus we see that fixing all the parameters but the standard deviation always gives a density satisfying the conditions.
\end{lem}

\begin{proof}
We calculate the density $f_m$ of $X_m$ by differentiating the cumulative distribution function $F_m$: \bea F_m(x_m) & \ = \ & \int_{x_{m-1}=0}^\infty {\rm Prob}(X_m \le x_m | X_{m-1} = x_{m-1}) {\rm Prob}(X_{m-1}=x_{m-1}) dx_{m-1} \nonumber\\ &=& \int_{x_{m-1}=0}^\infty {\rm Prob}(X_m \le x_m | X_{m-1} = x_{m-1}) f_{m-1}(x_{m-1})dx_{m-1} \nonumber\\ &=& \int_{x_{m-1}=0}^\infty \left[
\int_{t=0}^{x_m} f\left(\frac{t}{x_{m-1}}\right) \frac{dt}{x_{m-1}}  \right]
f_{m-1}(x_{m-1})dx_{m-1} \nonumber\\ &=& \int_{x_{m-1}=0}^\infty F\left(\frac{x_m}{x_{m-1}}\right) f_{m-1}(x_{m-1})dx_{m-1} \nonumber\\
f_m(x_m) &=& \int_{x_{m-1}=0}^\infty \frac1{x_{m-1}} f\left(\frac{x_m}{x_{m-1}}\right) f_{m-1}(x_{m-1})dx_{m-1} \nonumber\\ &=& \int_{x_{m-1}=0}^\infty f\left(\frac{x_m}{x_{m-1}}\right) f_{m-1}(x_{m-1}) \frac{dx_{m-1}}{x_{m-1}}. \eea
\end{proof}


We state two important special cases of Theorem \ref{thm:main}.

\begin{cor}\label{cor:main} Let the notation be as in Theorem \ref{thm:main}, and assume all conditions there are satisfied. \bi \item If $p(m) = 1$ for all $m$ (in other words, if we always use the same distribution), then \bea & &  {\rm Prob}(Y_n \bmod 1 \in [a,b]) - (b-a) \nonumber\\ & & \ \ \ \ \ \le \  (b-a) \cdot \left|\sum_{\ell = -\infty \atop \ell \neq 0}^\infty \left((\mathcal{M} f_{\mathcal{D}_{1}(1)})\left(1-\frac{2\pi i \ell}{\log B}\right)\right)^n\right|.\eea \item Let $\mathcal{D}_{{\rm Benf},B}$ be the distribution with density \be \twocase{f_{{\rm Benf},B}(x)\ =\ }{\frac1{x \log B}}{if $x\in [1,B)$}{0}{otherwise.} \ee Note if $X \sim \mathcal{D}_{{\rm Benf},B}$ then $X$ is Benford base $B$ (this follows by direct integration). If $\mathcal{D}_{p(1)}(1) = \mathcal{D}_{{\rm Benf},B}$ then for all $n$, $X_n$ is exactly Benford base $B$.

\ei
\end{cor}

Finally, we give a simple generalization of Theorem \ref{thm:main}.

\begin{cor}\label{cor:generalizemainthm} Notation and conditions as in Theorem \ref{thm:main}, for each $m \ge 1$ let $r(m)$ be a non-zero integer. Let now $X_m \sim \mathcal{D}_{p(m)}(X_{m-1}^{r(m-1)})$. Then the results of Theorem \ref{thm:main} still hold, except now $\left|{\rm Prob}(Y_n \bmod 1 \in [a,b]) - (b-a)\right|$ is at most \be \left|(b-a) \cdot \sum_{\ell = -\infty \atop \ell \neq 0}^\infty \prod_{m=1}^n (\mathcal{M} f_{\mathcal{D}_{p(m)}(1)})\left(1-\frac{2\pi i r(m) \ell}{\log B}\right)\right|. \ee
\end{cor}

\begin{rek} In Corollary \ref{cor:generalizemainthm} we could take $r(m) \in \Q - \{0\}$, and the proof would follow similarly. We chose to take $r(m) \in \Z-\{0\}$ as then the claim in Corollary \ref{cor:main} also holds. \end{rek}

We prove our main results in \S\ref{sec:proofsmainresults}, and comment on some alternate interpretations of our results. In particular, we show we may interpret our results in terms of the distribution of products of independent random variables, which has been connected to Benford's law by many authors (see the description and references in \cite{MN} for additional details).

One of our goals in this work is to demonstrate the ease of using the Mellin transform to obtain rapidly converging estimates on deviations from Benford's law. To this end we give some examples in \S\ref{sec:examples} where we only use one distribution in the chain, obtaining very rapidly converging (in $n$) bounds.

The proof of Corollary \ref{cor:generalizemainthm} follows from Theorem \ref{thm:main} and a lemma on the Mellin transform of the density of $X_{m-1}^{r(m-1)}$, which we give in Appendix \ref{app:Xr}. This is but one of many possible generalizations which can readily be studied using our methods.

Our results immediately apply to the situation of hierarchical Bayesian models with each variable depending on just one other variable. Thus we have established a connection between this field and Benford's Law. In particular, we see that when there are many levels then the observed sample values should approximately follow Benford's law, and thus these simple digit frequency tests can be used to test some detailed assumptions about hierarchical Bayesian models. In practice there is excellent agreement with Benford's law even when there are few levels; see the examples in \S\ref{sec:examples} for explicit bounds from uniform and exponential chains as well as examples where such chains may arise. In future work we plan to explore the case of chaining several variables, in order to handle the most general situations; for example, in addition to varying the scale, we will investigate the effects of changing the shape parameters of a distribution (such as the exponent in a Weibull family).


\section{Proof of Theorem \ref{thm:main}}\label{sec:proofsmainresults}

We first prove Theorem \ref{thm:main}, and then show how Corollary \ref{cor:main} follows.

\begin{proof}[Proof of Theorem \ref{thm:main}] We first calculate $f_n$, the density of $X_n$. The basis case is clear, and for the inductive step we note \bea f_n(x_n) \ = \ \int_0^\infty f_{\mathcal{D}_{p(n)}(1)}\left(\frac{x_n}{x_{n-1}}\right) f_{n-1}(x_{n-1}) \frac{dx_{n-1}}{x_{n-1}} \ = \ (f_{\mathcal{D}_{p(n)}(1)} \star f_{n-1})(x_n). \eea By the Mellin convolution theorem and induction we have \bea (\mel f_n)(s) & \ = \ & (\mel (f_{\mathcal{D}_{p(n)}(1)} \star f_{n-1}))(s)\nonumber\\ & \ = \ & (\mel f_{\mathcal{D}_{p(n)}(1)})(s) \cdot (\mel f_{n-1})(s) \nonumber\\ &=& \prod_{m=1}^n (\mel f_{\mathcal{D}_{p(m)}(1)})(s). \eea By the Mellin inversion theorem we find \bea\label{eq:eqfordensityfnmellininvmellin} f_n(x_n) &\ = \ & \left(\mel^{-1}\left( \prod_{m=1}^n (\mel f_{\mathcal{D}_{p(m)}(1)}(\cdot))\right)\right)(x_n). \eea

To investigate the distribution of the digits of $X_n$ (base $B$) it is convenient to make a logarithmic change of variables. Thus set $Y_n = \log_B X_n$. We have \bea {\rm Prob}(Y_n \le y) & \ = \ & {\rm Prob}(X_n \le B^y) \ = \ F_n(B^y). \eea Taking the derivative gives the density of $Y_n$, which we denote by $g_n(y)$: \be g_n(y) \ = \ f_n(B^y) B^y \log B. \ee A standard method to show $X_n$ tends to Benford behavior as $n\to\infty$ is to show that $Y_n \bmod 1$ tends to the uniform distribution on $[0,1]$ (see for example \cite{Dia,MT-B}). This can be seen from the following calculation. The key ingredient is Poisson Summation. While the argument is similar to that in \cite{KonMi}, the resulting expressions are not in the form considered there, and we thus cannot simply quote their results (though a trivial modification of that argument suffices). Let $h_{n,y}(t) = g_{n}(y+t)$. Then
\begin{equation}\label{eq:PoissonSummation}
\sum_{\ell=-\infty}^\infty g_{n}\left(y+\ell\right)\ =\
\sum_{\ell=-\infty}^\infty h_{n,y}(\ell)\ =\ \sum_{\ell=-\infty}^\infty \widehat{h}_{n,y}(\ell)\ =\ \sum_{\ell=-\infty}^\infty e^{2\pi
i y \ell}\widehat{g}_{n}(\ell),
\end{equation} where $\widehat{f}$ denotes the Fourier transform of $f$: \be \widehat{f}(\xi) \ = \ \int_{-\infty}^\infty f(x) e^{-2\pi i x \xi}dx. \ee
Letting $[a,b] \subset [0,1]$, we see that \bea {\rm Prob}(Y_n \bmod 1 \in [a,b]) & \ = \ & \sum_{\ell=-\infty}^\infty \int_{a+\ell}^{b+\ell} g_n(y) dy \nonumber\\ &=& \int_a^b \sum_{\ell=-\infty}^\infty g_n(y+\ell) dy \nonumber\\ &=& \int_a^b \sum_{\ell=-\infty}^\infty e^{2\pi i y \ell} \widehat{g}_n(\ell) dy \nonumber\\ &=& b-a + \err{(b-a) \sum_{\ell\neq 0} |\widehat{g}_n(\ell)|}.\eea Note that since $g_n$ is a probability density, $\widehat{g}_n(0) = 1$. The proof is completed by showing that the sum over $\ell$ tends to zero as $n \to \infty$. We thus need to compute $\widehat{g}_n(\ell)$: \bea \widehat{g}_n(\xi) & \ = \ & \int_{-\infty}^\infty g_n(y) e^{-2\pi i y \xi} dy \nonumber\\ &=& \int_{-\infty}^\infty f_n(B^y) B^y \log B \cdot e^{-2\pi i y \xi} dy \nonumber\\ &=& \int_0^\infty f_n(t) t^{-2\pi i \xi / \log B} dt \nonumber\\ &=& (\mel f_n)\left(1 - \frac{2\pi i \xi}{\log B}\right) \nonumber\\ &=& \prod_{m=1}^n(\mel f_{\mathcal{D}_{p(m)}(1)})\left(1 - \frac{2\pi i \xi}{\log B}\right). \eea Substituting completes the proof. \end{proof}

\begin{rek} If $f$ is a continuous density function, then $(\mel f)\left(1 - \frac{2\pi i \xi}{\log B}\right) < 1$ if $\xi \neq 0$. This is because $f(x)$ is non-negative and \be (\mel f)\left(1 - \frac{2\pi i \xi}{\log B}\right) \ = \ \int_0^\infty f(t) e^{-2\pi i \xi \log_B t} dt; \ee note the integral is clearly at most $\int_0^\infty f(t)dt = 1$ (since $f$ is a density) and in fact is less than this because of the oscillation due to the exponential factor. As $|\xi|$ grows this integral tends to zero rapidly. This follows from our assumption that the Mellin transform is a nice function, and indicates that we have rapid convergence if all the distributions in the chain are equal. An alternate proof of the decay in $|\xi|$ is to note that $(\mel f)\left(1 - \frac{2\pi i \xi}{\log B}\right)$ is the Fourier transform of $g(u) = f(e^u) e^u$ at $\xi/\log B$, and this tends to zero by the Riemann-Lebesgue lemma. \end{rek}

The above proof suggests the following:

\begin{cor}\label{cor:rearranging} Let $\sigma$ be a permutation of $\N$ (thus $\sigma$ is a 1-1 and onto map from $\N$ to $\N$). Assume all conditions in Theorem \ref{thm:main} hold for both some map $p:\N\to\N$ (with the chained random variables $X_m$) and $p\circ \sigma:\N\to\N$ (with the chained random variables $\widetilde{X}_m$). If $\{p(1),\dots,p(n)\} = \{p(\sigma(1)),\dots,p(\sigma(n))\}$ then the density of $X_n$ equals that of $\widetilde{X}_n$. \end{cor}

\begin{proof} The proof is immediate, and follows from the commutativity of multiplication in the expansion for the density $f_n$ in \eqref{eq:eqfordensityfnmellininvmellin}. \end{proof}

\begin{rek}\label{rek:interpretprodrandvar} The proof of Theorem \ref{thm:main} suggests another interpretation. Namely, the density of $X_n$ is exactly that of the density of $\Xi_1 \cdots \Xi_n$, where the $\Xi_m$ are independent random variables with $\Xi_m\sim \mathcal{D}_{p(m)}(1)$. For example, the density of the random variable $\Xi_1 \cdot \Xi_2$ is given by \be \int_0^\infty f_{\mathcal{D}_{p(2)}(1)}\left(\frac{x}{t}\right) f_{\mathcal{D}_{p(1)}(1)}(t) \frac{dt}{t}  \ee (the generalization to more products is straightforward). To see this, we first calculate the probability that $\Xi_1 \cdot \Xi_2 \in [0,x]$ and then differentiate with respect to $x$. Thus \bea  {\rm Prob}(\Xi_1 \cdot \Xi_2 \in [0,x]) & \ = \ & \int_{t=0}^\infty {\rm Prob}\left(\Xi_2 \in \left[0,\frac{x}t\right]\right) f_{\mathcal{D}_{p(1)}(1)}(t) dt \nonumber\\ &=& \int_{t=0}^\infty F_{\mathcal{D}_{p(2)}(1)}\left(\frac{x}{t}\right) f_{\mathcal{D}_{p(1)}(1)}(t) dt. \eea Differentiating gives the density of $\Xi_1 \cdot \Xi_2$, which equals \be \int_{t=0}^\infty f_{\mathcal{D}_{p(2)}(1)}\left(\frac{x}{t}\right) f_{\mathcal{D}_{p(1)}(1)}(t) \frac{dt}{t}. \ee Thus the convergence to Benford behavior of $X_n$ is equivalent to the convergence to Benford behavior of the product of $n$ identically distributed random variables. This is basically the central limit theorem for random variables modulo $1$ (see for example \cite{MN}), and thus yields an alternate proof of this important result (at least in this special case). Note this also gives another explanation for Corollary \ref{cor:rearranging}. \end{rek}

\begin{proof}[Proof of Corollary \ref{cor:main}] The first part follows immediately from Theorem \ref{thm:main}. For the second claim, we need the Mellin transform of $f_{{\rm Benf},B}$: \bea (\mathcal{M}f_{{\rm Benf},B})(s) & \ = \ & \int_0^\infty f_{{\rm Benf},B}(x) x^s \frac{dx}{x} \nonumber\\ &=& \frac1{\log B} \int_1^B x^{s-2} dx \nonumber\\  &=& \begin{cases} 1 & \text{{\rm if $s=1$}}\\ \frac1{\log B} \frac{B^{s-1} - 1}{s-1} &\text{{\rm if $s\neq 1$}} \end{cases} \eea
Thus \be \twocase{(\mathcal{M}f_{{\rm Benf},B})\left(1 - \frac{2\pi i \ell}{\log B}\right) \ = \ }{1}{if $\ell = 0$}{0}{if $0 \neq \ell \in \Z$.} \ee

Earlier we showed \be (\mathcal{M}f_n)(s) \ = \ \prod_{m=1}^n (\mathcal{M}f_{\mathcal{D}_{p(m)}(1)})(s). \ee We are assuming that $\mathcal{D}_{p(1)}(1) = \mathcal{D}_{{\rm Benf},B}$, and thus when we evaluate at $s = 1-\frac{2\pi i \ell}{\log B}$ with $\ell \in \Z$, the only term which survives is when $\ell=0$. From the proof of Theorem \ref{thm:main} we have \bea {\rm Prob}(Y_n \bmod 1 \in [a,b]) & \ = \ & b-a + \err{(b-a) \sum_{\ell\neq 0} |\widehat{g}_n(\ell)|}, \eea where $Y_n = \log_B X_n$ and \be \widehat{g}_n(\xi) \ = \ (\mathcal{M}f_n)\left(1 - \frac{2\pi i \xi}{\log B}\right); \ee note $X_n$ is Benford base $B$ if and only if $Y_n \bmod 1$ is the uniform distribution. As $\widehat{g}_n(\ell) = 0$ if $0 \neq \ell \in \Z$ (from evaluating the Mellin transform of $f_1$), we obtain that \be {\rm Prob}(Y_n \bmod 1 \in [a,b]) \ = \ b-a; \ee thus $X_n$ is Benford base $B$ for all $n$. \end{proof}

\begin{rek} Note that, unlike the other theorems, we have Benford behavior for a finite value of $n$; there are no error terms. Further, by Corollary \ref{cor:rearranging}, we obtain that $X_n$ is exactly Benford base $B$ if for some $m \le n$ we have $X_m \sim \mathcal{D}_{{\rm Benf},B}(X_{m-1})$. \end{rek}


\section{Examples}\label{sec:examples}

We give two explicit examples of the types of rapidly converging error estimates easily obtainable from these methods. The first example is chaining exponential distributions. Many processes have wait times governed by a Poisson or exponential distribution; thus applications of these results could be to more involved processes where the wait time parameter depends on another process. For our second example we consider chaining uniform distributions. Our street example gives one instance where this could arise, namely when we choose uniformly among options of varying size.

\subsection{Chains of the Exponential Distribution}

Let $X_1 \sim \expd{1}$ (the standard exponential distribution) and $X_m \sim \expd{X_{m-1}}$, and set $Y_m = \log_B X_m$. By Theorem \ref{thm:main} we know that as $n\to\infty$ the distribution of digits of $X_n$ tends to Benford's law; we now bound the error term. We need the following two ingredients: \bi \item the Mellin transform of the standard exponential function (which we denote by $f_{\exp}$) is the Gamma function: \bea \int_0^\infty \exp(-x) x^{s-1} dx \ = \ \Gamma(s). \eea Thus \be (\mel f_{\exp})\left(1-\frac{2\pi i \ell}{\log B}\right) \ = \ \Gamma\left(1 - \frac{2\pi i \ell}{\log B}\right). \ee \item for real $x$, \be \left|\Gamma(1+ix)\right| \ = \ \sqrt{\pi x / \sinh(\pi x)}. \ee \ei Substituting these into Theorem \ref{thm:main} (or Corollary \ref{cor:main}) gives
\bea {\rm Prob}(Y_n \bmod 1 \in [a,b]) & \ = \ &b-a + \err{(b-a) \sum_{\ell=1}^\infty \left(\frac{2\pi^2 \ell / \log B}{\sinh(2\pi^2 \ell/\log B)}\right)^{n/2}},\nonumber\\ \eea or equivalently the probability that the mantissa of $X_n$ is in $[1,s]$ is \be \log_B s + \err{\log_B s \sum_{\ell=1}^\infty \left(\frac{2\pi^2 \ell / \log B}{\sinh(2\pi^2 \ell/\log B)}\right)^{n/2}}.\ee As $\sinh(x)$ grows exponentially in $x$, we see the above sum converges rapidly (i.e., the large $\ell$ terms are immaterial), and the error term decreases rapidly with $n$.

If we take $B=10$ we find the difference between the probability of observing the mantissa of $X_n$ in $[1,s]$ and the Benford probability of $\log_B s$ is at most $.0033\log_B s$ if $n=2$, $.00019\log_B s$ if $n=3$, $.000011\log_B s$ if $n=5$ and
$3.6 \cdot 10^{-13} \log_B s$ if $n=10$. If $B=10$ then for all $\ell \ge 1$ we have $\exp(2\pi^2 \ell / \log 10) - \exp(-2\pi^2 \ell / \log 10)  \ge \frac{10000}{10001} \exp(2\pi^2 \ell / \log 10)$. Thus the error term is bounded by \be \log_{10} s \sum_{\ell =1}^\infty \left( \frac{17.148 \ell}{\exp(8.5726 \ell)} \right)^{n/2} \ \le \ .057^n \log_{10} s. \ee

\subsection{Chains of the Uniform Distribution}

Let $X_1\sim {\rm Unif}(0,k)$ (without loss of generality we may assume $k \in [1,10)$) and set $X_m \sim {\rm Unif}(0,X_{m-1})$. If $P_n(s)$ is
the probability that the base 10 mantissa of $X_n$ is at most $s$, then \bea P_n(s) \ = \ \log_{10} s + \err{\frac{k}{s} \frac{(\log k)^{n-1}}{\Gamma(n)} + \left(\frac1{2.9^n} + \frac{\zeta(n) - 1}{2.7^n} \right) 2 \log_{10}s }. \eea

As the uniform distribution is so easy to work with, we sketch an alternate, more explicit derivation; in fact, it was by generalizing this and the exponential case (which involved properties of the Meijer $G$-function) that led us to the proof of the general case. One can prove by induction that \be f_n(x_n) \ = \ \frac{\log^n(k/x_n)}{k\Gamma(n+1)}. \ee For the base case $n = 2$, since $X_1 \sim \unif{k}$ we have \bea F_{2,k}(x_2) & \ = \ & \int_0^k {\rm Prob}\left(X_2 \le x_2 | X_1 = x_{1}\right) {\rm Prob}(X_1 = x_1) dx_1 \nonumber\\ &=&   \int_0^{x_2} {\rm Prob}(X_2 \le x_2 | X_1 = x_1) \frac{dx_1}{k} +\int_{x_2}^k {\rm Prob}(X_2 \le x_2|X_1 = x_1) \frac{dx_1}{k} \nonumber\\ &=& \int_0^{x_2} \frac{dx_1}{k} + \int_{x_2}^k \frac{x_2}{x_1} \frac{dx_1}{k} \nonumber\\ & = & \frac{x_2}k + \frac{x_2 \log (k/x_2)}{k}. \eea Differentiating yields \be f_{2,k}(x_2) \ = \ \frac{\log(k/x_2)}{k}, \ee which proves the base case. The inductive step follows similarly.

We have \be P_n(s) \ = \ \sum_{\ell=1}^\infty \int_{10^{-\ell}}^{s\cdot 10^{-\ell}} f_{n,k}(x_n)dx_n + \int_1^{\min(s,k)} f_{n,k}(x_n)dx_n. \ee Note for large $n$ the contribution from the second integral is negligible, as the integrand is bounded by $(\log k)^{n-1} / (n-1)!$, which tends rapidly to 0 for fixed $k$ and increasing $n$. We change variables by letting $u = \log(k/x_n)$. Thus $du = -x_n^{-1} dx_n$ or $dx_n = k e^{-u} du$. Thus if we set \be \twocase{g_{n}(u) \ = \ }{\frac{u^{n-1} \ e^{-u}}{\Gamma(n)}}{if $u \ge 0$}{0}{if $u \le 0$,} \ee we find that \bea P_n(s) & \ = \ &   \sum_{\ell=-\infty}^\infty \int_{\log k +\ell \log 10 - \log s}^{\log k + \ell \log 10} g_{n}(u) du \ - \ \int_{\log k - \log s}^{\log k - \log(\min(s,k))} g_{n}(u)du, \nonumber\\  \eea where $g_{n}(u) = 0$ for $u \le 0$ allows us to extend the $\ell$-sum to all integers. The contribution from the second integral is negligible, as it is bounded by $\frac{k}{s} \frac{(\log k)^{n-1}}{\Gamma(n)}$. We evaluate the main term by Poisson Summation.
Thus \bea\label{eq:expPnsunifchain1} P_n(s) & \ = \   & \sum_{\ell=-\infty}^\infty \int_{\log k +\ell \log 10 - \log s}^{\log k + \ell \log 10} g_{n}(u) du + \err{\frac{k}{s} \frac{(\log k)^{n-1}}{\Gamma(n)}} \nonumber\\ &=& \int_{\log k - \log s}^{\log k} \sum_{\ell=-\infty}^\infty g_{n} (u + \ell \log 10) du + \err{\frac{k}{s} \frac{(\log k)^{n-1}}{\Gamma(n)}} \nonumber\\ &=&
\int_{\log_{10} k - \log_{10} s}^{\log_{10} k} \sum_{\ell=-\infty}^\infty g_{n}\left((w+\ell)\log 10\right)dw + \err{\frac{k}{s} \frac{(\log k)^{n-1}}{\Gamma(n)}} \nonumber\\ &=& \int_{\log_{10} k - \log_{10} s}^{\log_{10} k} \sum_{\ell=-\infty}^\infty h_{n,w}\left(\ell\right) \log 10\ dw + \err{\frac{k}{s} \frac{(\log k)^{n-1}}{\Gamma(n)}}, \eea where $h_{n,w}(t) = g_{n}((w+t)/T)$ with $T = 1/\log 10$. We have written our sum like this to facilitate applying the Poisson Summation formula. We have
\begin{equation}\label{eq:Poissum1}
\sum_{\ell=-\infty}^\infty g_{n}\left(\frac{w+\ell}{T}\right)\ =\
\sum_{\ell=-\infty}^\infty h_{n,w}(\ell)\ =\ \sum_{\ell=-\infty}^\infty \widehat{h}_{n,w}(\ell)\ =\ T\sum_{\ell=-\infty}^\infty e^{2\pi
i w \ell}\widehat{g}_{n}(T\ell).
\end{equation}
Recall that $g_n(u)$ is the density function for the Gamma distribution with parameter $n$. Its characteristic function is well-known to be $\E[e^{it}] = (1 - it)^{-n}$; thus its Fourier transform (which is $\E[e^{-2\pi i t}]$) is just $\widehat{g}_n(t) = (1+2\pi i t)^{-n}$. Therefore substituting \eqref{eq:Poissum1} into \eqref{eq:expPnsunifchain1} and splitting off the contribution from $\ell = 0$ yields \bea P_n(s) & \ = \ & \int_{\log_{10} k - \log_{10} s}^{\log_{10} k} \sum_{\ell=-\infty}^\infty e^{2\pi i w \ell} \left(1+\frac{2\pi i \ell}{\log 10}\right)^{-n}dw + \err{\frac{k}{s} \frac{(\log k)^{n-1}}{\Gamma(n)}} \nonumber\\ & = & \log_{10} s + \err{\frac{k}{s} \frac{(\log k)^{n-1}}{\Gamma(n)}} + \err{\sum_{\ell=1}^\infty 2 \log_{10}(s) \cdot \left(1+\frac{2\pi i \ell}{\log 10}\right)^{-n} }. \nonumber\\ \eea

The error term is easily analyzed. The contribution from $\ell=1$ is bounded by $(2.9)^{-n} 2 \log_{10} s$, while the $\ell \ge 2$ terms contribute at most \be \frac{2 \log_{10} s}{2.7^n} \sum_{\ell = 2}^\infty \ell^{-n} \ = \ \frac{2 \log_{10} s}{2.7^n} \cdot \left( \zeta(n) - 1\right), \ee where \be \zeta(s)\ = \ \sum_{\ell=1}^\infty \frac1{\ell^s}, \ \ \ \Re(s) > 1\ee is the Riemann zeta function. Thus \bea P_n(s) & \ = \ & \log_{10} s + \err{\frac{k}{s} \frac{(\log k)^{n-1}}{\Gamma(n)} + \left(\frac1{2.9^n} + \frac{\zeta(n) - 1}{2.7^n} \right) 2 \log_{10}s }.\ \ \ \ \ \eea

\ \\

\section*{Acknowledgements}

We thank Alex Ely Kossovsky for sharing his preprint, Eric Bradlow and Rick Cleary for very detailed readings of the manuscript, Christoph Leuenberger for helpful comments on an earlier draft (especially for suggesting the connection with hierarchical Bayesian models), and the participants of the Workshop on Applications of Benford's Law for many enlightening conversations. The last named author was partly supported by NSF grant DMS0600848.\\

\ \\


\appendix

\section{General Powers of Random Variables}\label{app:Xr}

Our main theorem considers a chain of random variables, where $X_m \sim \mathcal{D}_{p(m)}(X_{m-1})$. Our proof uses properties of the Mellin transform, and shows the equivalence of chaining to products of random variables.

More generally, for each $m$ let $r(m)$ be a non-zero integer. We consider now $X_m \sim \mathcal{D}_{p(m)}(X_{m-1}^{r(m-1)})$. Our theorems generalize immediately to this case as well. The key ingredient is the following lemma.

\begin{lem}\label{lem:randvarr} Let $W$ have density $\phi$, and for $r \in \Z-\{0\}$ let $U = W^r$ have density $\psi_r$. Then \bea \psi_r(u) & \ = \  & \frac1{|r|} \phi\left(u^{1/|r|}\right) u^{\frac{1-r}{r}} \nonumber\\ (\mel \psi_r)(s) &=& (\mel \phi)\left(r(s-1)+1\right). \eea In particular, taking $s = 1 - \frac{2\pi i \ell}{\log B}$ yields \be (\mel \psi_r)\left(1 - \frac{2\pi i \ell}{\log B}\right) \ = \ (\mel \phi)\left(1 - \frac{2\pi i r\ell}{\log B}\right). \ee
\end{lem}

\begin{proof} We calculate the cumulative distribution function of $U$, and then differentiate to get its density. We consider $r > 0$ (the case of $r = -|r| < 0$ is handled similarly). We have \bea \Psi_r(u) \ = \ {\rm Prob}(U \le u) \ = \  {\rm Prob}(W^r \le u) \ = \ {\rm Prob}\left(W \le u^{1/r}\right) \ = \ \Phi(u^{1/r}), \eea where $\Phi$ is the antiderivative of $\phi$. Thus \be \psi_r(u) \ = \ \frac1{r} \phi\left(u^{1/r}\right) u^{\frac{1-r}{r}}. \ee We now calculate the Mellin transform, again considering just the case of $r > 0$, as the other case follows similarly. \bea (\mel \psi_r)(s) & \ = \ & \int_0^\infty \psi_r(u) u^s \frac{du}{u} \nonumber\\ &=& \int_0^\infty \frac1{r} \phi\left(u^{1/r}\right) u^{\frac{1-r}{r}} u^s \frac{du}{u} \nonumber\\ &=& \int_0^\infty \phi(t) t^{r(s-1)} dt \nonumber\\ &=& \int_0^\infty \phi(t) t^{r(s-1)+1} \frac{dt}{t} \ = \ (\mel \phi)\left(r(s-1)+1\right);   \eea the remaining claim follows by direct substitution.
\end{proof}

\begin{rek} For us, one of the most important consequences of Lemma \ref{lem:randvarr} is that when we evaluate the resulting Mellin transform at $1 - \frac{2\pi i \ell}{\log B}$ we end up with the Mellin transform of another density evaluated at $1 - \frac{2\pi i r(m-1) \ell}{\log B}$.  Thus our arguments from before follow with almost no change; it is essential that the effect of replacing $X_{m-1}$ with $X_{m-1}^{r(m-1)}$ is only to change the \emph{imaginary} part of where we evaluate. We could take $r(m-1) \in \Q-\{0\}$ or even $\R - \{0\}$ and the argument would still hold (but now the second part of Corollary \ref{cor:main} would fail).
\end{rek}



\bigskip

\ \\

\end{document}